\documentclass[11pt]{article}
\usepackage{multibib}
\usepackage{ulem}
\usepackage{float}
\usepackage[export]{adjustbox}
\usepackage{graphicx}
\graphicspath{ {./images/} }
\usepackage{subcaption}
\usepackage[utf8]{inputenc}
\usepackage[export]{adjustbox}
\usepackage{wrapfig}
\usepackage{authblk}
\usepackage[top=27mm, bottom=27mm, left=22mm, right=22mm]{geometry}
\usepackage{amsthm,amsmath,amssymb}
\usepackage{mathtools}
\usepackage{mathrsfs}
\usepackage{multirow}
\usepackage{microtype}
\usepackage{hyperref}
\usepackage[capitalize]{cleveref}
\usepackage{xcolor}
\usepackage{verbatim}
\usepackage{indentfirst}
\usepackage{cite}

\usepackage{ulem}

\newtheorem{theorem}{Theorem}[section]
\newtheorem{lemma}[theorem]{Lemma}
\newtheorem{proposition}[theorem]{Proposition}

\newtheorem{definition}[theorem]{Definition}

\newtheorem{example}[theorem]{Example}
\newtheorem{question}[theorem]{Question}

\newtheorem{fact}[theorem]{Fact}

\newcommand{\nos}{\operatorname{N}}

\title{Incidence theorems for multivariate polynomials over finite fields}
\author{}
\date{}

\begin{document}

\author{Chong Shangguan\thanks{C. Shangguan is with Research Center for Mathematics and Interdisciplinary Sciences, Shandong University, Qingdao 266237, China, and Frontiers Science Center for Nonlinear Expectations, Ministry of Education, Qingdao 266237, China (Email: theoreming@163.com)}, Yulin Yang\thanks{Y. Yang is with Research Center for Mathematics and Interdisciplinary Sciences, Shandong University, Qingdao 266237, China (Email: forestyoung@mail.sdu.edu.cn)} and Tao Zhang\thanks{T. Zhang is with Institute of Mathematics and Interdisciplinary Sciences, Xidian University, Xi'an 710071, China (Email: zhant220@163.com)}
}

\maketitle

\begin{abstract}
\noindent We study incidence problems for multivariate polynomials over a finite field $\mathbb{F}_q$. Given two families of $m$-variate polynomials, we count the number of triples $(f,g,x)$ such that $f$ belongs to the first family, $g$ belongs to the second family, $x\in\mathbb{F}_q^m$, and $f(x)=g(x)$. We show that for any subsets $\mathcal{L},\mathcal{L}'\subseteq V_{m,r}$, where $V_{m,r}$ denotes the vector space of all $m$-variate polynomials over $\mathbb{F}_q$ of degree at most $r$, the number of such triples is at most
$$q^{m-1}|\mathcal{L}||\mathcal{L}'|+O\big(q^{\dim V_{m,r}-1}\sqrt{|\mathcal{L}||\mathcal{L}'|}\big).$$
We further show that if $\mathcal{L}$ and $\mathcal{L}'$ are contained in a subspace $V\subseteq V_{m,r}$ satisfying a suitable separating condition, then the same estimate holds with $\dim V_{m,r}$ replaced by $\dim V$. Our upper bound is essentially sharp when $q^{m-1}|\mathcal{L}||\mathcal{L}'|$ dominates the summation.

As applications, we derive incidence bounds for points and multivariate polynomials. These results recover and strengthen several previously known bounds for point-line incidences and point-univariate-polynomial incidences.

Our proof is spectral, relying on an expander mixing lemma for general abelian Cayley color graphs together with Fourier analysis over finite fields.
\end{abstract}

\noindent\textbf{Keywords:} incidence problem in finite fields; multivariate polynomials; Cayley color graphs; expander mixing lemma

\noindent\textbf{MSC2020:} 11T06; 51E30; 05B30

\section{Incidence theorems}\label{sec:1}

\subsection{Incidence theorems for points and multivariate polynomials}

\noindent Incidence problems are a central theme of combinatorial and discrete geometry. A classical starting point is the Szemer\'edi--Trotter theorem, which states that for a set of points $\mathcal{P}$ and a set of lines $\mathcal{L}$ in the real plane $\mathbb{R}^2$, the number of incidences
$I(\mathcal{P},\mathcal{L})=|\{(v,\ell)\in\mathcal{P}\times\mathcal{L}:v\in\ell\}|$
is bounded by
$O\big(|\mathcal{P}|^{2/3}|\mathcal{L}|^{2/3}+|\mathcal{P}|+|\mathcal{L}|\big).$
This result has been generalized in many directions, including to other fields and to incidences with higher-degree curves; see, for example, \cite{bourgain2004sum,clarkson1990combinatorial,kong2025point,mohammadi2023point,pach1998number,phuong2016incidence,phuong2017incidences,rudnev2018number,spencer1984unit,tamo2025points,toth2015szemeredi}. It has also found many applications in combinatorics \cite{iosevich2004fourier,spencer1984unit,tamo2025points}, geometry \cite{guth2015erdHos,sharir2018distinct}, and number theory \cite{bourgain2004sum,elekes1997number}.

Over finite fields, the situation is rather different. The direct analogue of the Szemer\'edi--Trotter theorem fails over $\mathbb{F}_q^2$: all $q^2$ points and all $q^2+q$ lines determine $q^3+q^2$ incidences, which is of order $(|\mathcal{P}||\mathcal{L}|)^{3/4}$. Seeking finite-field analogues of Szemer\'edi--Trotter type bounds, Bourgain, Katz, and Tao \cite{bourgain2004sum} proved that if $\mathcal{P}$ is a set of $N$ points and $\mathcal{L}$ is a set of $N$ lines in $\mathbb{F}_p^2$, where $p$ is prime and $N=p^\alpha$ for some $0<\alpha<2$, then the number of incidences is at most $O(N^{3/2-\varepsilon})$, where $\varepsilon>0$ depends only on $\alpha$. Their result has since been improved in various parameter ranges; see, for example, \cite{helfgott2011explicit,iosevich2024improved,jones2011explicit,jones2016improved,rudnev2022growth,stevens2017improved,petridis2022energy,vinh2011szemeredi}.

A particularly clean finite-field incidence theorem was obtained by Vinh \cite{vinh2011szemeredi}. For a set of points $\mathcal{P}$ and a set of lines $\mathcal{L}$ in $\mathbb{F}_q^2$, he proved that
\begin{equation}\label{eq:Vinh-point-line}
    \left|I(\mathcal{P},\mathcal{L})-\frac{|\mathcal{P}||\mathcal{L}|}{q}\right|
    \le q^{1/2}\sqrt{|\mathcal{P}||\mathcal{L}|}.
\end{equation}
Tamo \cite{tamo2025points} extended this direction to point-polynomial incidences. For a set of points $\mathcal{P}\subseteq\mathbb{F}_q^2$ and a set of univariate polynomials $\mathcal{L}$ of degree at most $r$, he showed that
\begin{equation}\label{eq:Tamo-original}
    \left|I(\mathcal{P},\mathcal{L})-\frac{|\mathcal{P}||\mathcal{L}|}{q}\right|
    \leq \sqrt{|\mathcal{P}||\mathcal{L}|(q+|\mathcal{L}|(r-1))}.
\end{equation}
More recently, Arala and Chow \cite{arala2025expansion} proved the bound
\begin{equation}\label{eq:Tamo-point-polynomial}
    \left|I(\mathcal{P},\mathcal{L})-\frac{|\mathcal{P}||\mathcal{L}|}{q}\right|
    \le q^{r/2}\sqrt{|\mathcal{P}||\mathcal{L}|}.
\end{equation}
This result was also implicitly proved by Mattheus, Mubayi, Nie and Verstra\"ete \cite{mattheus2025off}. When $r=1$, \eqref{eq:Tamo-original} and \eqref{eq:Tamo-point-polynomial} coincide. When $r\ge 2$, the two bounds are incomparable. 


As applications of our main result (see \cref{thm:cross-version}), below we present incidence theorems for points and multivariate polynomials, generalizing and strengthening both \eqref{eq:Vinh-point-line} and \eqref{eq:Tamo-point-polynomial}. 

Let $V_{m,r}=\{f\in\mathbb{F}_q[x_1,\ldots,x_m]:\deg(f)\le r\}$ denote the set of multivariate polynomials in $m$ variables over $\mathbb{F}_q$, with degree at most $r$. Clearly, $V_{m,r}$ is an $\binom{m+r}{r}$-dimensional vector space over $\mathbb{F}_q$.
For a set of points $\mathcal{P}\subseteq\mathbb{F}_q^{m+1}$ and a set of polynomials $\mathcal{L}\subseteq V_{m,r}$, we say that a point $v=(v_1,\ldots,v_{m+1})\in\mathcal{P}$ is {\it incident to} a polynomial $f\in\mathcal{L}$ if $f(v_1,\ldots,v_m)=v_{m+1}$. Let $I(\mathcal{P},\mathcal{L})=|\{(v,f)\in\mathcal{P}\times\mathcal{L}:f(v_1,\ldots,v_m)=v_{m+1}\}|$ denote the number of incidences between $\mathcal{P}$ and $\mathcal{L}$. 

\begin{proposition}\label{prop:point-multi-poly}
    For every set of points $\mathcal{P}\subseteq\mathbb{F}_q^{m+1}$ and every set of multivariate polynomials $\mathcal{L}\subseteq V_{m,r}$, 
    \begin{equation}\label{eq:point-multi-poly}
        \left|I(\mathcal{P},\mathcal{L})-\frac{|\mathcal{P}||\mathcal{L}|}{q}\right|\le q^{(\dim V_{m,r}-1)/2}\sqrt{|\mathcal{P}||\mathcal{L}|}.
    \end{equation}
\end{proposition}

\noindent Plugging $\dim V_{1,1}=2$ and $\dim V_{1,r}=r+1$ into \eqref{eq:point-multi-poly} recover \eqref{eq:Vinh-point-line} and \eqref{eq:Tamo-point-polynomial}, respectively. Plugging $\dim V_{m,1}=m+1$ into \eqref{eq:point-multi-poly} recovers a point-hyperplane incidence bound of Vinh (see \cite[Theorem 4]{vinh2011szemeredi}). \footnote{For point-line incidences, in contrast to \eqref{eq:Vinh-point-line}, the bounds \eqref{eq:Tamo-point-polynomial} and \eqref{eq:point-multi-poly} do not consider vertical lines. Including vertical lines in $\mathcal{L}$ will increase the right-hand sides of \eqref{eq:Tamo-point-polynomial} and \eqref{eq:point-multi-poly} by at most $|\mathcal{P}|$ incidences, since every point is incident to at most one vertical line. Nevertheless, Tamo \cite{tamo2025points} showed that one can fix this issue by a random affine transformation argument, and obtained a modest improvement over \eqref{eq:Vinh-point-line} (see \cite[Proposition 1.4]{tamo2025points}).}  

It is natural to ask that if the set of polynomials $\mathcal{L}$ is contained in a subspace $V$ of $V_{m,r}$, then could one use $\dim V$ to replace $\dim V_{m,r}$ in \eqref{eq:point-multi-poly}? In general, this is not true, as shown by the following example. 

\begin{example}\label{ex:counter-example}
    Consider the set of points $\mathcal{P}_0=\{(\alpha,0):\alpha\in\mathbb{F}_q^m\}$ and the set of multivariate polynomials $\mathcal{L}_0=\{x_1g(x_1,\ldots,x_m):g\in V_{m,r-1}\}$. Observe that $\mathcal{L}_0$ is a subspace of $V_{m,r}$ and 
    $$\frac{|\mathcal{P}_0||\mathcal{L}_0|}{q}+q^{(\dim \mathcal{L}_0-1)/2}\sqrt{|\mathcal{P}_0||\mathcal{L}_0|}=q^{m-1}|\mathcal{L}_0|+q^{(m-1)/2}|\mathcal{L}_0|.$$ 
It is not hard to see that $I(\mathcal{P}_0,\mathcal{L}_0)=(2-\frac{1}{q})q^{m-1}|\mathcal{L}_0|$ (for completeness, see Appendix~\ref{sec-app-A} for a proof), which is roughly twice as large as the above value for $m\ge 2$ and large $q$. 
\end{example}

We show that if the subspace $V$ of $V_{m,r}$ satisfies certain {\it separating} condition, then one can use $\dim V$ instead of $\dim V_{m,r}$ in \eqref{eq:point-multi-poly}. It gives considerably sharper bounds when $\dim V\ll\dim V_{m,r}$. 

\begin{definition}[Separating subspace]\label{def:sep-subspace}
    We say that a subspace $V$ in $\mathbb{F}_q[x_1,\ldots,x_m]$ is a separating subspace if $V$ has a basis of monomials $\{x_1^{i_1}\cdots x_m^{i_m}:(i_1,\ldots,i_m)\in\mathcal{I}\}$, where $\mathcal{I}$ is a finite subset of $\mathbb{N}^m$, such that $(0,0,\ldots,0)\in\mathcal{I}$ and, for any distinct $\alpha,\beta\in\mathbb{F}_q^m$, there exists $(i_1,\ldots,i_m)\in\mathcal{I}$ satisfying $\alpha_1^{i_1}\cdots\alpha_m^{i_m}\neq \beta_1^{i_1}\cdots\beta_m^{i_m}$.
\end{definition}


\begin{example}\label{def:property-*}
Let $V$ be a subspace in $\mathbb{F}_q[x_1,\ldots,x_m]$ with a basis of monomials $\{x_1^{i_1}\cdots x_m^{i_m}:(i_1,\ldots,i_m)\in\mathcal{I}\}$, where $\mathcal{I}$ is a finite subset of $\mathbb{N}^m$ satisfying the following conditions:
    \begin{itemize}
        \item $(0,0,\ldots,0)\in\mathcal{I}$;
        \item there exist positive integers $k_1,k_2,\ldots,k_m$ such that $(k_1,0,\ldots,0),(0,k_2,\ldots,0),\ldots,(0,0,\ldots,k_m)\in\mathcal{I}$, with each $k_i$ satisfying $\gcd(k_i,q-1)=1$.
    \end{itemize}
Then $V$ is a separating subspace. Indeed, if $\alpha,\beta\in\mathbb{F}_q^m$ are distinct, then $\alpha_j\neq\beta_j$ for some $j$; since $\gcd(k_j,q-1)=1$, the map $t\mapsto t^{k_j}$ is injective on $\mathbb{F}_q$, and hence $\alpha_j^{k_j}\neq\beta_j^{k_j}$.
\end{example}

\begin{theorem}\label{thm:subspace-point-multi-poly}
    Let $V$ be a separating subspace of $V_{m,r}$. Then for every $\mathcal{P}\subseteq\mathbb{F}_q^{m+1}$ and $\mathcal{L}\subseteq V$,
    \begin{equation}\label{eq:subspace-point-multi-poly}
        \left|I(\mathcal{P,\mathcal{L}})-\frac{|\mathcal{P}||\mathcal{L}|}{q}\right|\leq q^{(\dim V-1)/2}\sqrt{|\mathcal{P}||\mathcal{L}|}.
    \end{equation}
\end{theorem}

We have several remarks regarding \eqref{eq:subspace-point-multi-poly}. First, \eqref{eq:point-multi-poly} follows easily from \eqref{eq:subspace-point-multi-poly}, since $V_{m,r}$ is separating. Second, \eqref{eq:subspace-point-multi-poly} shows that when $V$ is separating and $|\mathcal{P}||\mathcal{L}|\gg q^{\dim V+1}$, we have $I(\mathcal{P,\mathcal{L}})=(1+o(1))\frac{|\mathcal{P}||\mathcal{L}|}{q}$. Third, in many cases (e.g., $|\mathcal{P}| > q$ and $|\mathcal{L}| > \frac{1}{r} q^{\dim V - m}$) \eqref{eq:subspace-point-multi-poly} improves the upper bound on $I(\mathcal{P,\mathcal{L}})$ obtained by the Cauchy-Schwarz inequality.

\begin{equation}\label{eq:point-multi-poly-CS}
    I(\mathcal{P},\mathcal{L})\leq\min\left\{|\mathcal{L}|+q^{\dim V_{m,r}/2-1}|\mathcal{P}||\mathcal{L}|^{1/2}, |\mathcal{P}|+r^{1/2}q^{(m-1)/2}|\mathcal{P}|^{1/2}|\mathcal{L}|\right\}.
\end{equation}

\noindent (For completeness, we include a proof of \eqref{eq:point-multi-poly-CS} in Appendix~\ref{sec:app-B}). Fourth, \eqref{eq:subspace-point-multi-poly} also implies a special case of a recent point-variety incidence bound of Kong and Tamo (see $d=1$ case in \cite[Theorem 1.2]{kong2025point}). 

Next, we shift from point-polynomial incidence bounds to polynomial-polynomial incidence bounds. The latter are more general and essentially imply \eqref{eq:subspace-point-multi-poly}.

\subsection{Incidence theorems for two sets of multivariate polynomials}

\noindent For a multivariate polynomial $f\in\mathbb{F}_q[x_1,\ldots,x_m]$, let $\nos_q(f)=|\{\alpha\in\mathbb{F}_q^m:f(\alpha)=0\}|$ denote the number of zeros of $f$ in $\mathbb{F}_q^m$. For two polynomials $f,f'\in\mathbb{F}_q[x_1,\ldots,x_m]$, it is natural to regard $\nos_q(f-f')=|\{\alpha\in\mathbb{F}_q^m:f(\alpha)=f'(\alpha)\}|$ as the number of {\it incidences} between $f$ and $f'$.  

Counting the number of zeros of polynomials over finite fields is a classic and vital topic in combinatorics and number theory. It is well-known by the Schwartz-Zippel lemma that for two distinct polynomials $f,f'\in V_{m,r}$, $\nos_q(f-f')\le rq^{m-1}$. However, if one generates $f,f'$ uniformly and independently at random from $V_{m,r}$, then, in expectation, $\nos_q(f-f')=q^{m-1}$. The gap between the {\it worst case} and the {\it average case} is a factor of $r$. It is interesting to study to what extent we can {\it bridge} such a gap. 

We show that for every sufficiently large subset $\mathcal{L}\subseteq V_{m,r}$, say $|\mathcal{L}|\gg V_{m,r}/q^{m}$, and two randomly chosen polynomials $f,f'\in\mathcal{L}$, $f-f'$ has in expectation $(1+o(1))q^{m-1}$ zeros. 

\begin{proposition}\label{prop:incidence-bound-multi-poly}
    For every set of multivariate polynomials $\mathcal{L}\subseteq V_{m,r}$, we have
    \begin{align}\label{eq:incidence-bound-multi-poly}
        q^{m-1}|\mathcal{L}|^2\le \sum_{f,f'\in\mathcal{L}}\nos_q(f-f')\le q^{m-1}|\mathcal{L}|^2+q^{\dim V_{m,r}-1}|\mathcal{L}|.
    \end{align}
\end{proposition}

It follows from \eqref{eq:incidence-bound-multi-poly} that for $|\mathcal{L}|\gg q^{\dim V_{m,r}-m}$, we have $\sum_{f,f'\in\mathcal{L}} \nos_q(f-f')=(1+o(1))q^{m-1}|\mathcal{L}|^2$. The above lower bound on $|\mathcal{L}|$ is essentially the best possible for $m=1$ or $r=1$, as shown by the aforementioned $\mathcal{L}_0=\{x_1g(x_1,\dots,x_m):g\in V_{m,r-1}\}$ in Example~\ref{ex:counter-example}. Indeed, it is easy to see that $|\mathcal{L}_0|=q^{\dim V_{m,r-1}}$ and $\sum_{f,f'\in \mathcal{L}_0} \nos_q(f-f')=(2-\frac{1}{q})q^{m-1}|\mathcal{L}_0|^2$, and for $m=1$ or $r=1$ one has $\dim V_{m,r}-m= \dim V_{m,r-1}$. It is an open question to determine for each $m\ge 2,r\ge 2$, the smallest $\tau_q(m,r)$, so that for every $\mathcal{L}\subseteq V_{m,r}$ with $|\mathcal{L}|\gg \tau_q(m,r)$, $\sum_{f,f'\in\mathcal{L}} \nos_q(f-f')=(1+o(1))q^{m-1}|\mathcal{L}|^2$. 

Similarly to the strengthening \eqref{eq:subspace-point-multi-poly} upon \eqref{eq:point-multi-poly}, we have the following theorem which implies \eqref{eq:incidence-bound-multi-poly}.

\begin{theorem}\label{thm:subpace-incidence-bound-multi-poly}
    Let $V$ be a separating subspace of $V_{m,r}$. Then for every $\mathcal{L}\subseteq V$, we have
    \begin{align}\label{eq:subpace-incidence-bound-multi}
        q^{m-1}|\mathcal{L}|^2\le\sum_{f,f'\in\mathcal{L}}\nos_q(f-f')\le q^{m-1}|\mathcal{L}|^2+q^{\dim V-1}|\mathcal{L}|.
    \end{align}
\end{theorem}

Lastly, we present our most general incidence theorem, which is a cross version of \eqref{eq:subpace-incidence-bound-multi}.

\begin{theorem}\label{thm:cross-version}
    Let $V$ be a separating subspace of $V_{m,r}$. Then for every $\mathcal{L},\mathcal{L}'\subseteq V$, we have
    \begin{equation}\label{eq:cross-version}
        \left|\sum_{f\in\mathcal{L},f'\in\mathcal{L'}}\nos_q(f-f')-q^{m-1}|\mathcal{L}||\mathcal{L'}|\right|\leq q^{\dim V-1}\sqrt{|\mathcal{L}|\mathcal{|L'|}}.
    \end{equation}
\end{theorem}

\noindent Clearly, the upper bound in \eqref{eq:subpace-incidence-bound-multi} follows from \eqref{eq:cross-version} by setting $\mathcal{L}=\mathcal{L}'$.

\paragraph{Outline of the proof.} To prove \cref{thm:cross-version}, we introduce a natural yet previously unexplored connection between incidence problems and a specially defined Cayley color graph, which we refer to as the {\it polynomial incidence graph} (see Definition~\ref{def:Cay(V_m,r,N_q)}). 

The vertex set of this graph is $V_{m,r}$ and each directed edge $(f,f')\in V_{m,r}\times V_{m,r}$ is assigned a color corresponding to the number of incidences $\nos_q(f-f')$. By interpreting incidences between polynomials as edge colors, we bound the number of incidences between two sets of polynomials via the total weight of the colored edges (see Definition~\ref{def:e(S,T)}). 

We then prove an expander mixing lemma for general abelian Cayley color graphs (see Lemma~\ref{lem:expander-mixing-Cayley}), which estimates the total weight of colored edges between two vertex subsets via the graph spectrum. This result generalizes the classical expander mixing lemma by Alon and Chung  \cite{alon1994explicit} for ordinary graphs. 

The key technical difficulty of this work is the complete determination of the spectrum of the polynomial incidence graph using discrete Fourier analysis over finite fields. An immediate application of the expander mixing lemma then yields our main result Theorem~\ref{thm:cross-version}, from which both Theorem~\ref{thm:subpace-incidence-bound-multi-poly} and Theorem~\ref{prop:incidence-bound-multi-poly} follow as straightforward corollaries. 

Lastly, motivated by the work of Murphy and Petridis \cite{murphy2016point}, we use a second moment argument to bound point-polynomial incidences via polynomial-polynomial incidences and prove Theorem~\ref{thm:subspace-point-multi-poly}. 

At a high level, our proof adopts a spectral approach similar to those in \cite{vinh2011szemeredi} and \cite{tamo2025points}. However, compared with \cite{vinh2011szemeredi} and \cite{tamo2025points}, our proof has three new ingredients. First, rather than working with the point-polynomial incidence graph, we move a step further by considering the polynomial incidence graph, which allows us to prove incidence bounds for two sets of multivariate polynomials. Second, to establish our main result, we develop an expander mixing lemma for abelian Cayley color graphs, which is of independent interest and may have further applications. 
Third, we obtain an improved error term when the polynomials in $\mathcal{L}$ are chosen from a separating subspace of $V_{m,r}$ (see Definition~\ref{def:sep-subspace}).

\paragraph{Organization.} The remaining part of this paper is organized as follows. In \cref{sec:2}, we introduce Cayley color graphs and formally define the polynomial incidence graph. We also discuss its connection to incidence problems. We then present an expander mixing lemma for abelian Cayley color graphs (see Lemma~\ref{lem:expander-mixing-Cayley}) and characterize the spectrum of the polynomial incidence graph (see \cref{thm:spectrum}). In \cref{sec:3}, after reviewing some basics of representation theory and Fourier analysis over finite fields, we present the proofs of Lemma~\ref{lem:expander-mixing-Cayley} and Theorem~\ref{thm:spectrum}. In \cref{sec:4}, we present the proofs of our main results, Theorems \ref{thm:cross-version} and \ref{thm:subpace-incidence-bound-multi-poly}.
Finally, in \cref{sec:5}, we mention several open problems for future research.

\section{Cayley color graphs and its connection with incidence theorems}
\label{sec:2}

\noindent Originally introduced by Arthur Cayley \cite{cayley1878desiderata}, Cayley graphs have become fundamental structures in modern graph theory. Given a group $G$ and a subset $S \subseteq G$, the {\it Cayley graph} of $G$ with respect to the {\it connection set} $S$, denoted by $\operatorname{Cay}(G, S)$, is the graph with vertex set $G$ where two vertices $x$ and $y$ are adjacent if and only if $xy^{-1}\in S$. Cayley graphs have numerous applications in combinatorics and theoretical computer science. To name a few, they are used in the deterministic constructions of pseudorandom clique-free graphs \cite{alon1994explicit,bishnoi2020construction} and Ramanujan graphs \cite{lubotzky1988ramanujan}.

As a generalization of Cayley graphs, to the best of our knowledge, Cayley color graphs were first formally defined by Babai \cite{babai1979spectra} in 1979. 


\begin{definition}[Cayley color graphs]\label{def:Cayley}
    Let $G$ be a finite group and $c:G\to\mathbb{C}$ a function. The {\it Cayley color graph} of $G$ with {\it connection function} $c$, denoted by $\operatorname{Cay}(G,c)$, is the directed graph with vertex set $G$ and edge set $G\times G$, where each directed edge $(x,y)$ is colored by $c(xy^{-1})$. The {\it adjacency matrix} of $\operatorname{Cay}(G,c)$ is the matrix with rows and columns indexed by the elements of $G$, whose $(x,y)$-entry is $c(xy^{-1})$.
\end{definition}

Building on the work of Lov\'asz \cite{lovasz1975spectra}, Babai \cite{babai1979spectra} derived an expression for the spectrum of the Cayley color graph $\operatorname{Cay}(G, c)$ in terms of the irreducible characters of the group $G$.

To incorporate Cayley color graphs with incidence bounds, below we define the {\it weight} of {\it colored edges} in a Cayley color graph. 

\begin{definition}[Weight of colored edges]\label{def:e(S,T)}
    Let $\operatorname{Cay}(G,c)$ be a Cayley color graph. Then for two subsets $S,T\subseteq G$, define the {\it weight} of {\it colored edges} directed from $S$ to $T$ as $e_c(S,T)=\sum_{(x,y)\in S\times T}c(xy^{-1})$.
\end{definition}

Next, we define a special family of Cayley color graphs, which we call the {\it polynomial incidence graph}.

\begin{definition}[Polynomial incidence graph]\label{def:Cay(V_m,r,N_q)}
    Let $\operatorname{Cay}(V_{m,r},\nos_q)$ be the Cayley color graph defined on the additive group $V_{m,r}$ with connection function $\nos_q:V_{m,r}\rightarrow\mathbb{N}$, defined as for each $f\in V$, $\nos_q(f)=|\{\alpha\in\mathbb{F}_q^m:f(\alpha)=0\}|$. 
    $\operatorname{Cay}(V_{m,r},\nos_q)$ is an undirected complete graph, where each edge $\{f,f'\}\subseteq V$ is colored by $\nos_q(f-f')$.
    The {\it adjacency matrix} of $\operatorname{Cay}(V_{m,r},\nos_q)$ is a symmetric matrix with rows and columns indexed by the elements of $V_{m,r}$, whose $(f,f')$-entry is $\nos_q(f-f')$.
\end{definition}

We have the following crucial observation, which connects Cayley color graphs to incidence theorems.

\begin{fact}\label{fact:crucial-equality}
    By Definition~\ref{def:e(S,T)}, for every $\mathcal{L},\mathcal{L}'\subseteq V_{m,r}$, $e_{\nos_q}(\mathcal{L},\mathcal{L}')=\sum_{f\in\mathcal{L},f'\in\mathcal{L'}}\nos_q(f-f')$.
\end{fact}

Therefore, to prove \eqref{eq:cross-version} (and all the other incidence bounds in this paper), it suffices to prove corresponding upper and lower bounds on $e_{\nos_q}(\mathcal{L},\mathcal{L}')$. Such a bound has been proved for every ordinary regular graph with large spectral gap. The celebrated {\it expander mixing lemma} (see \cite[Lemma 2.3]{alon1988explicit}) showed that for every $d$-regular $n$-vertex graph $H$, if the absolute values of any other eigenvalues of $H$ except $d$ are at most $\lambda$, then for every $S,T\subseteq V(H)$,
\begin{equation}\label{eq:expander-mixing}
    \left|e(S,T)-\frac{d}{n}|S||T|\right|\leq\lambda\sqrt{|S||T|\left(1-\frac{|S|}{n}\right)\left(1-\frac{|T|}{n}\right)}.
\end{equation}

Fact~\ref{fact:crucial-equality} motivates us to prove an expander mixing lemma for Cayley color graphs, as stated below.

\begin{lemma}[Expander mixing lemma for abelian Cayley color graphs]\label{lem:expander-mixing-Cayley}
    Let $G$ be a finite abelian group and $c:G\to\mathbb{C}$ be a function. Let $\Gamma=\operatorname{Cay}(G,c)$ be the Cayley color graph with vertex set $G$ and connection function $c$. Then for every $S, T\subseteq V(\Gamma)$, 
    \begin{equation}\label{eq:expander-mixing-Cayley}
        \left|e_c(S,T)-\frac{1}{|G|}\sum_{g\in G}c(g)|S||T|\right|\leq\lambda\sqrt{|S||T|\left(1-\frac{|S|}{|G|}\right)\left(1-\frac{|T|}{|G|}\right)},
    \end{equation}
    
    \noindent where $\lambda=\max\{|\widehat{c}(\chi)|:\chi\in\widehat{G},\chi\neq\chi_0\}$ is the maximum absolute value of the non-trivial Fourier coefficients of $c$.
\end{lemma}

We have several remarks regarding \eqref{eq:expander-mixing-Cayley}. First, setting $T=\{1_G\}$, the identity element in $G$, \eqref{eq:expander-mixing-Cayley} implies that 
\begin{equation*}
    \left|\frac{1}{|S|}\sum_{x\in S} c(x)-\frac{1}{|G|}\sum_{g\in G}c(g)\right|\le\frac{\lambda}{\sqrt{|S|}},
\end{equation*}

\noindent which gives an upper bound on the deviation between $\mathbb{E}_{g\sim S}[c(x)]$ and $\mathbb{E}_{g\sim G}[c(x)]$. Second, by considering the trace of the square of the adjacency matrix of $\operatorname{Cay}(G,c)$, we give a lower bound $\lambda\geq\sqrt{|G|\operatorname{Var}_{g\sim G}|c(g)|}$ (see concluding remarks for details). 

To prove the desired lower and upper bounds on  $\sum_{f\in\mathcal{L},f'\in\mathcal{L'}}\nos_q(f-f')$ via Fact~\ref{fact:crucial-equality} and \eqref{eq:expander-mixing-Cayley}, we determine the spectrum of the polynomial incidence graph $\operatorname{Cay}(V_{m,r},\nos_q)$, as shown by the following result. Note that $p$ is the characteristic of the finite field $\mathbb{F}_q$ and $\mathbb{F}_q^\ast =\mathbb{F}_q\setminus\{0\}$; $\zeta_p=e^{\frac{2\pi i}{p}}$ is the $p$-th root of unity; $\operatorname{Tr}$ is the trace function from $\mathbb{F}_q$ to $\mathbb{F}_p$ (see \cref{subsec:pre} below for details).

\begin{theorem}[Spectrum of the polynomial incidence graph]\label{thm:spectrum}
    Let $V$ be a separating subspace of $\mathbb{F}_q[x_1,\ldots,x_m]$ as defined in Definition~\ref{def:sep-subspace}, and $\Gamma=\operatorname{Cay}(V,\nos_q)$ be the polynomial incidence graph. Then the eigenvalues of $\Gamma$ are $q^{\dim V+m-1}, q^{\dim V-1}$ and $0$, with multiplicities $1$, $(q-1)q^m$, and $|V|-(q-1)q^m-1$, respectively. Moreover, the eigenspace corresponding to $q^{\dim V+m-1}$ is spanned by the all one vector and the eigenspace corresponding to $q^{\dim V-1}$ is spanned by $\{(\zeta_p^{\operatorname{Tr}(Cf(\alpha))}:f\in V):C\in\mathbb{F}_q^\ast, \alpha\in\mathbb{F}_q^m\}$.
\end{theorem}

The proof of \cref{thm:spectrum} is the main technical part of this paper. As we will see in the next section, the eigenvalues of $\operatorname{Cay}(V,\nos_q)$ can be expressed as character sums over $V_{\alpha}$, where for each $\alpha\in\mathbb{F}_q^m$, $V_{\alpha}$ consists of all polynomials $f\in V$ vanishing at $\alpha$. By the first orthogonality relation of characters, this reduces to the determination of the annihilators of $V_{\alpha}$ for all $\alpha\in\mathbb{F}_q^n$. The reader is referred to \cref{subsec:pre} and \cref{subsec:spectrum} for details.


\section{Spectral bounds for Cayley color graphs}
\label{sec:3}

\noindent The goal of this section is to present the proofs of Lemma~\ref{lem:expander-mixing-Cayley} and Theorem~\ref{thm:spectrum}. 

\subsection{Preliminaries}\label{subsec:pre}

\noindent We begin by introducing some basic concepts from the representation theory of finite abelian groups (for details, see, e.g. \cite{steinberg2012representation}).

Let $G$ be a finite abelian group written multiplicatively, and let $1_G$ denote the identity element of $G$. A {\it character} of $G$ is a homomorphism from $G$ to the multiplicative group of complex numbers $\mathbb{C}^*$; that is, a map $\chi:G\to\mathbb{C}^*$ such that $\chi(gh)=\chi(g)\chi(h)$ for all $g,h\in G$. Since $g^{|G|}=1_G$ for all $g\in G$, it follows that $\chi(g)^{|G|}=\chi(g^{|G|})=\chi(1_G)=1$. Hence, the values of $\chi$ are $|G|$-th roots of unity. In particular, $\chi(g^{-1})=\chi(g)^{-1}=\overline{\chi(g)}$, where the bar denotes complex conjugation. 

Let $\widehat{G}$ denote the set of all characters of $G$, called the {\it dual} of $G$. The set $\widehat{G}$ forms an abelian group under pointwise multiplication; that is, for $\chi, \psi \in \widehat{G}$, define $(\chi \psi)(g) = \chi(g) \psi(g)$ for all $g \in G$. The identity element of $\widehat{G}$ is the {\it trivial character}, denoted by $\chi_0$, defined by $\chi_0(g) = 1$ for all $g \in G$. The inverse of $\chi \in \widehat{G}$ is given by the {\it conjugate character} $\overline{\chi}$, which is defined by $\overline{\chi}(g)=\overline{\chi(g)}$ for all $g\in G$. An important property of finite abelian groups is {\it self-duality}, which states that $\widehat{G} \cong G$. Hence, we can label the characters of $G$ as $\widehat{G} = \{\chi_g : g \in G\}$.

Next, let us recall the {\it additive characters} of vector spaces over finite fields. Most of the results can be found in \cite{lidl1997finite}.

\begin{example}[Additive characters of finite fields]\label{eg:vector-space-dual-group}
    Let $d,s$ be positive integers, $p$ be a prime, and $q=p^s$. Let $\mathbb{F}_q$ be the finite field of order $q$, and let $\mathbb{F}_q^d$ be the vector space of dimension $d$ over $\mathbb{F}_q$.
    
    Consider the additive group $\mathbb{F}_q^d$, the characters in $\widehat{\mathbb{F}_q^d}=\{\chi_\alpha:\alpha\in\mathbb{F}_q^d\}$ are given by 
    $\chi_\alpha(\beta)=\zeta_p^{\operatorname{Tr}(\langle\alpha,\beta\rangle)}$ for every $\beta \in \mathbb{F}_q^d$, where $\zeta_p=e^{\frac{2\pi i}{p}}$, $\operatorname{Tr}:\mathbb{F}_{q}\to\mathbb{F}_p$, defined by $\operatorname{Tr}(a)=a+a^p+\cdots+a^{p^{s-1}}$ is the trace function from $\mathbb{F}_q$ to the prime field $\mathbb{F}_p$ and $\langle \alpha,\beta\rangle:=\sum_{i=1}^d\alpha_i\beta_i$ is the inner product of two vectors in $\mathbb{F}_q^d$. 
\end{example}

Let $H$ be a subgroup of $G$. The {\it annihilator} of $H$, denoted by $H^\perp$, is the set of all characters $\chi\in\widehat{G}$ such that $\chi(h)=1$ for all $h\in H$; in other words, the restriction of $\chi$ to $H$ is the trivial character of $H$. It is well known that the structure of $H^\perp$ is given by:

\begin{fact}[Theorem~5.6 in \cite{lidl1997finite}]\label{fact:annihilator}
Let $H$ be a subgroup of the finite abelian group $G$. 
Then the annihilator of $H$ is a subgroup of $\widehat{G}$ of order $|G|/|H|$.
\end{fact}

Let $\mathbb{C}^{G}$ denote the set of all functions from $G$ to $\mathbb{C}$. Crucially, each function $f\in\mathbb{C}^{G}$ is equivalently viewed as a column vector indexed by elements of $G$, i.e., $f=(f(g):g\in G)\in\mathbb{C}^{|G|}$. The {\it first orthogonality relation} of characters states that for $\chi,\psi\in\widehat{G}$,
\begin{equation}\label{eq:first-orthogonality}
    \sum_{g\in G}\chi(g)\overline{\psi(g)}=
    \begin{cases}|G|, & \chi=\psi,\\
    0,              & \chi\neq\psi.
    \end{cases}
\end{equation}

\noindent Using \eqref{eq:first-orthogonality}, it is not hard to see that $\widehat{G}$ forms an orthonormal basis of $\mathbb{C}^G$ under the inner product
\begin{equation*}
    \langle f_1,f_2\rangle_{\mathbb{C}^G}\coloneqq\frac{1}{|G|}\sum_{g\in G}f_1(g)\overline{f_2(g)}.
\end{equation*}


By the orthogonality of the characters, one has $f=\sum_{\chi\in \widehat{G}}\langle f,\chi\rangle_{\mathbb{C}^G} \chi$ for all $f\in\mathbb{C}^G$. The {\it Fourier transform} of $f$, is the function $\widehat{f}:\widehat{G}\to\mathbb{C}$, defined by $\widehat{f}(\chi)=|G|\langle f,\chi\rangle_{\mathbb{C}^G}=\sum_{g\in G}f(g)\overline{\chi(g)}$, where $\widehat{f}(\chi),~ \chi\in\widehat{G}$ are called the {\it Fourier coefficients} of $f$. In particular, $\widehat{f}(\chi_0)$ is referred to as the {\it trivial Fourier coefficient} of $f$. Hence, $f=\frac{1}{|G|}\sum_{\chi\in \widehat{G}}\widehat{f}(\chi)\chi$, and the formula for the {\it inverse Fourier transform} is
$f(g)=\frac{1}{|G|}\sum_{\chi\in\widehat{G}}\widehat{f}(\chi)\chi(g)$. The inner product of the Fourier transform $\widehat{f_1},\widehat{f_2}\in\mathbb{C}^{\widehat{G}}$ of $f_1,f_2$ satisfies the {\it Plancherel formula}: $\langle \widehat{f_1},\widehat{f_2}\rangle_{\mathbb{C}^{\widehat{G}}}=|G|\langle f_1,f_2\rangle_{\mathbb{C}^G}$.


It is well known from the results of Lov\'asz \cite{lovasz1975spectra} and Babai \cite{babai1979spectra} that the spectrum of the Cayley color graph $\operatorname{Cay}(G,c)$ is given by the Fourier coefficients of its connection function. We include its short proof for completeness.

\begin{lemma}[Corollary 3.2 in \cite{babai1979spectra}]\label{lem:EigCayley}
    Let $G$ be a finite abelian group and $c:G\to\mathbb{C}$ be a function. Then the eigenvalues of $\operatorname{Cay}(G,c)$ are 
    \begin{equation*}
        \lambda_{\chi}=\widehat{c}(\chi),~\chi\in \widehat{G}.
    \end{equation*}

\noindent Moreover, the eigenvector corresponding to $\lambda_\chi$ is $\chi=(\chi(g):g\in G)\in\mathbb{C}^{|G|}$.
\end{lemma}

\begin{proof}
    Let $A$ be the adjacency matrix of $\operatorname{Cay}(G,c)$. Since for every $\chi\in\widehat{G}$ and $x\in G$, the $x$-coordinate of the vector $A\chi$ is
    \begin{equation*}
        (A\chi)(x)=\sum_{y\in G}c(xy^{-1})\chi(y)=\sum_{g\in G}c(g)\chi(g^{-1}x)=\left(\sum_{g\in G}c(g)\overline{\chi(g)}\right)\chi(x)=\widehat{c}(\chi) \chi(x),
    \end{equation*}

    \noindent we have $A\chi=\widehat{c}(\chi)\chi$, as needed.
\end{proof}

\subsection{The expander mixing lemma: Proof of Lemma \ref{lem:expander-mixing-Cayley}}

\begin{proof}[Proof of Lemma \ref{lem:expander-mixing-Cayley}]



Let $A$ be the adjacency matrix of $\operatorname{Cay}(G, c)$. Let $1_X$ denote the indicator vector of a subset $X\subseteq G$. Let $v^*$ denote the conjugate transpose of a complex vector $v$. Then,
\begin{equation*}    
\begin{split}
    e_c(S,T)
    &=\sum_{(x,y)\in S\times T}c(xy^{-1})=1_S^\ast\cdot A\cdot 1_T^{}=\left(\frac{1}{|G|}\sum_{\chi\in\widehat{G}}\widehat{1_S}(\chi)\chi\right)^* A\left(\frac{1}{|G|}\sum_{\chi\in\widehat{G}}\widehat{1_T}(\chi)\chi\right)\\
    &=\left(\frac{1}{|G|}\sum_{\chi\in\widehat{G}}\widehat{1_S}(\chi)\chi\right)^* \left(\frac{1}{|G|}\sum_{\chi\in\widehat{G}}\widehat{1_T}(\chi)\widehat{c}(\chi)\chi\right)=\frac{1}{|G|^2}\sum_{\chi\in\widehat{G}}\widehat{1_S}(\chi)\widehat{1_T}(\chi)\widehat{c}(\chi)|G|\\
    &=\frac{1}{|G|}\left(\widehat{c}(\chi_0)\widehat{1_S}(\chi_0)\widehat{1_T}(\chi_0)+\sum_{\chi\in\widehat{G}\setminus\{\chi_0\}}\widehat{c}(\chi)\widehat{1_S}(\chi)\widehat{1_T}(\chi)\right),\\
\end{split}
\end{equation*}

\noindent where the first three equalities follow from definition, the fourth equality follows from the spectrum of $\operatorname{Cay}(G,c)$ given by Lemma~\ref{lem:EigCayley}, and the fifth equality follows from the orthogonality of characters \eqref{eq:first-orthogonality}.

Since $\widehat{c}(\chi_0)=\sum_{g\in G}c(g),~\widehat{1_S}(\chi_0)=|S|$ and $\widehat{1_T}(\chi_0)=|T|$, we rearrange the above equation as

\begin{equation*}
\begin{split}
    \left|e_c(S,T)-\frac{1}{|G|}\sum_{g\in G}c(g)|S||T|
    \right|
    &=\frac{1}{|G|}\left|\sum_{\chi\in\widehat{G}\setminus\{\chi_0\}}\widehat{c}(\chi)\widehat{1_S}(\chi)\widehat{1_T}(\chi)\right|\\
    &\leq\frac{\lambda}{|G|}\sum_{\chi\in\widehat{G}\setminus\{\chi_0\}}\left|\widehat{1_S}(\chi)\widehat{1_T}(\chi)\right|\\
    &\leq\lambda\sqrt{\left(\frac{1}{|G|}\sum_{\chi\in\widehat{G}\setminus\{\chi_0\}}|\widehat{1_S}(\chi)|^2\right)\left(\frac{1}{|G|}\sum_{\chi\in\widehat{G}\setminus\{\chi_0\}}|\widehat{1_T}(\chi)|^2\right)}\\
    &=\lambda\sqrt{\left(\langle \widehat{1_S},\widehat{1_S}\rangle_{\mathbb{C}^{\widehat{G}}}-\frac{1}{|G|}|\widehat{1_S}(\chi_0)|^2\right)\left(\langle \widehat{1_T},\widehat{1_T}\rangle_{\mathbb{C}^{\widehat{G}}}-\frac{1}{|G|}|\widehat{1_T}(\chi_0)|^2\right)}\\
    &=\lambda\sqrt{|S||T|\left(1-\frac{|S|}{|G|}\right)\left(1-\frac{|T|}{|G|}\right)},\\
\end{split}
\end{equation*}

\noindent where the second inequality follows from the Cauchy-Schwarz inequality and the last equality follows from the Plancherel formula $\langle\widehat{f_1},\widehat{f_2}\rangle_{\mathbb{C}^{\widehat{G}}}=|G|\langle f_1,f_2\rangle_{\mathbb{C}^G}$ and $\langle 1_X,1_X\rangle_{\mathbb{C}^G}=\frac{|X|}{|G|}$ for all $X\subseteq G$. 
\end{proof}

\subsection{Spectrum of the polynomial incidence graph: Proof of \cref{thm:spectrum}}\label{subsec:spectrum}

\noindent Throughout this subsection, $V$ is always a separating subspace of $\mathbb{F}_q[x_1,\ldots,x_m]$. Since $V$ is isomorphic to $\mathbb{F}_q^{\dim V}$ as an $\mathbb{F}_q$-linear space, we henceforth identify each polynomial $f(x_1,\dots,x_m)=\sum_{(i_1,\dots,i_m)\in\mathcal{I}}f_{(i_1,\dots,i_m)}x_1^{i_1}\dots x_m^{i_m}\in V$ as a vector $(f_{(i_1,\dots,i_m)}:(i_1,\dots,i_m)\in\mathcal{I})\in\mathbb{F}_q^{\dim V}$. 
Hence $\widehat{V}\cong\widehat{\mathbb{F}_q^{\dim V}}$, and by Example~\ref{eg:vector-space-dual-group}, we have $\widehat{V}\cong\{\chi_f:f\in V\}$, where $\chi_f(g)=\zeta_p^{\operatorname{Tr}(\langle f,g\rangle)}$ for all $g\in V$. Note that for $f,g\in V$, $\langle f,g\rangle=\sum_{(i_1,\dots,i_m)\in\mathcal{I}}f_{(i_1,\dots,i_m)}g_{(i_1,\dots,i_m)}$. For each $\alpha\in\mathbb{F}_q^m$, let $V_\alpha=\{f\in V:f(\alpha)=0\}$.

Below we prove the key technical lemma of this paper.



\begin{lemma}[Key technical lemma]\label{lem:vanishing-subspace}
Let $V\subseteq \mathbb{F}_q[x_1,\ldots,x_m]$ and $\mathcal{I}\subseteq\mathbb{N}^m$ be defined in Definition~\ref{def:sep-subspace}. Then
\begin{itemize}
    \item [1)] For each $\alpha\in\mathbb{F}_q^m$, we have $|V_{\alpha}|=q^{\dim V-1}$.
    \item [2)] For every $C\in\mathbb{F}_q$ and $\alpha\in\mathbb{F}_q^m$, let     \begin{equation*}
        p_{C,\alpha}(x_1,\ldots,x_m)=C\cdot\sum_{(i_1,\ldots,i_m)\in\mathcal{I}}\alpha_1^{i_1}\cdots\alpha_m^{i_m}x_1^{i_1}\cdots x_{m}^{i_m}.
    \end{equation*}
    Then $V_{\alpha}^\perp=\{\chi_{p_{C,\alpha}}:C\in\mathbb{F}_q\}$.
    \item [3)] For every two distinct $\alpha,\beta\in\mathbb{F}_q^m$, we have $V_{\alpha}^{\perp}\cap V_\beta^{\perp}=\{\chi_0\}$.
\end{itemize}
\end{lemma}

\begin{proof}
    \begin{itemize}
        \item [1)] Consider the linear map $E_\alpha:V\to\mathbb{F}_q$, where each $f\in V$ is mapped to $E_\alpha(f)=f(\alpha)$. 
        Then $V_{\alpha}=\ker E_{\alpha}$. By the rank-nullity theorem, to prove the lemma, it suffices to show that $E_\alpha(V)\neq\{0\}$. Indeed, since $(0,\ldots,0)\in \mathcal{I}$, $V$ contains the constant function $1$ that maps everything to $1$, in particular, $1\in E_{\alpha}(V)$. 
        \item [2)]     Since $\langle p_{C,\alpha },f\rangle=C\cdot\sum_{i\in\mathcal{I}}f_i\alpha^i=C\cdot f(\alpha )$, 
        we have $\chi_{p_{C,\alpha }}(f)=\zeta_p^{\operatorname{Tr}(\langle p_{C,\alpha },f\rangle)}=\zeta_p^{\operatorname{Tr}(C\cdot f(\alpha))}$. 
    
        On one hand, we have $\{\chi_{p_{C,\alpha }}:C\in\mathbb{F}_q\}\subseteq V_{\alpha}^{\perp }$, since for each $f\in V_{\alpha}$, $\chi_{p_{C,\alpha}}(f)=\zeta_p^{\operatorname{Tr}(C\cdot f(\alpha))}=\zeta_p^0=1$. On the other hand, by Fact~\ref{fact:annihilator} and 1), we have $|V_\alpha^\perp |=|V|/|V_\alpha|=q$. Since $|\{\chi_{p_{C,\alpha}}:C\in\mathbb{F}_q\}|=q$, we conclude that $V_\alpha^\perp =\{\chi_{p_{C,\alpha}}:C\in\mathbb{F}_q\}$.
        \item [3)] Observe that for $C=0$, $\chi_{p_{0,\alpha}}=\chi_{p_{0,\beta}}=\chi_0$. Therefore, $\chi_0\in V_{\alpha}^\perp\cap V_\beta^\perp$. To prove the lemma, it suffices to show that $\{p_{C,\alpha}:C\in\mathbb{F}_q^\ast \}\cap\{p_{C,\beta}:C\in\mathbb{F}_q^\ast \}=\varnothing$. 
        Clearly, for every two distinct $C,C'\in\mathbb{F}_q^\ast $, $p_{C,\alpha}\neq p_{C',\beta}$ since $(p_{C,\alpha}-p_{C',\beta})(0)=C-C'\neq0$. It remains to show that $p_{1,\alpha}\neq p_{1,\beta}$. Since $V$ is a separating subspace and $\alpha\neq\beta$, there exists $(i_1,\ldots,i_m)\in\mathcal{I}$ such that $\alpha_1^{i_1}\cdots\alpha_m^{i_m}\neq\beta_1^{i_1}\cdots\beta_m^{i_m}$. Hence the coefficient of $x_1^{i_1}\cdots x_m^{i_m}$ in $p_{1,\alpha}-p_{1,\beta}$ is nonzero, and so $p_{1,\alpha}\neq p_{1,\beta}$. 
    \end{itemize}
\end{proof}

Finally, we are in a position to present the proof of \cref{thm:spectrum}.

\begin{proof}[Proof of \cref{thm:spectrum}]
    Let $A$ be the adjacency matrix of the polynomial incidence graph $\Gamma=\operatorname{Cay}(V,\nos_q)$. By Lemma~\ref{lem:EigCayley}, the eigenvalues of $A$ are $\widehat{\nos_q}(\chi), \chi\in\widehat{V}$. Since $A$ is a real symmetric matrix, it has real eigenvalues.  Therefore, for each $\chi\in\widehat{V}$,     \begin{equation*}
        \begin{split}
            \widehat{\nos_q}(\chi)
            =\overline{\widehat{\nos_q}(\chi)}
            &=\sum_{f\in V}\nos_q(f)\chi(f)\\
            &=\sum_{f\in V}\sum_{\beta \in\mathbb{F}_q^m:f(\beta )=0}\chi(f)
            =\sum_{\beta \in\mathbb{F}_q^m}\sum_{f\in V_\beta}\chi(f).\\
        \end{split}
    \end{equation*}

\noindent Since $V_{\beta}$ is a subgroup of $V$, the restriction of the characters of $V$ to $V_{\beta}$ are also characters of $V_{\beta}$. Then, by the first orthogonality relation of the characters of $V_{\beta}$, we have
\begin{equation}\label{eq:sub-space-orthogonality}
    \sum_{f\in V_\beta}\chi(f)=
    \begin{cases}
    |V_{\beta}|, & \chi\in V_{\beta}^\perp,\\
    0          , & \chi\notin V_{\beta}^\perp.
    \end{cases}
\end{equation}

\noindent Therefore, by \eqref{eq:sub-space-orthogonality} and Lemma~\ref{lem:vanishing-subspace}, it is not hard to verify the folllowing equalites.
\begin{itemize}
    \item If $\chi=\chi_0$, then
        \begin{equation*}
            \widehat{\nos_q}(\chi_0)
            =\sum_{\beta\in\mathbb{F}_q^m}\sum_{f\in V_\beta}\chi_0(f)
            =\sum_{\beta\in\mathbb{F}_q^m}\sum_{f\in V_\beta}1
            =q^{\dim V-1+m};\\
        \end{equation*}
        
    \item If $\chi=\chi_{p_{C,\alpha}}$ for $C\in\mathbb{F}_q^\ast  $ and $\alpha\in\mathbb{F}_q^m$, then 
    \begin{equation*}
        \begin{split}
            \widehat{\nos_q}(\chi_{p_{C,\alpha}})
            &=\sum_{\beta\in\mathbb{F}_q^m}\sum_{f\in V_\beta}\chi_{p_{C,\alpha}}(f)\\
            &=\sum_{f\in V_\alpha}\chi_{p_{C,\alpha}}(f)+\sum_{\beta\in\mathbb{F}_q^m:\beta\neq\alpha}\sum_{f\in V_\beta}\chi_{p_{C,\alpha}}(f)\\
            &=|V_\alpha|+0=q^{\dim V-1};
        \end{split}
    \end{equation*}
    
    \item If $\chi\notin\{\chi_{p_{C,\alpha}}:C\in\mathbb{F}_q,\alpha\in\mathbb{F}_q^m\}$, then
        \begin{equation*}
            \widehat{\nos_q}(\chi)=\sum_{\beta\in\mathbb{F}_q^m}\sum_{f\in V_\beta}\chi(f)=0.
        \end{equation*}
\end{itemize}
Hence by Lemma~\ref{lem:EigCayley}, the eigenvalues of $\Gamma$ are $q^{\dim V+m-1}, q^{\dim V-1}$ and $0$, with multiplicities $1$, $(q-1)q^m$, and $|V|-(q-1)q^m-1$, respectively. Moreover, the eigenspace corresponding to $q^{\dim V+m-1}$ is spanned by the all one vector and the eigenspace corresponding to $q^{\dim V-1}$ is spanned by $\{(\zeta_p^{\operatorname{Tr}(Cf(\alpha))}:f\in V):C\in\mathbb{F}_q^\ast , \alpha\in\mathbb{F}_q^m\}$.
\end{proof}

\section{Incidence theorems for multivariate polynomials}
\label{sec:4}

\noindent The main goal of this section is to present the proofs of Theorems~\ref{thm:cross-version},~\ref{thm:subpace-incidence-bound-multi-poly} and~\ref{thm:subspace-point-multi-poly}.

\subsection{Multivariate polynomial incidences: Proofs of Theorems~\ref{thm:cross-version} and~\ref{thm:subpace-incidence-bound-multi-poly}}

\begin{proof}[Proof of \cref{thm:cross-version}]
    To prove \cref{thm:cross-version}, we apply the expander mixing lemma (Lemma~\ref{lem:expander-mixing-Cayley}) to the polynomial incidence graph $\Gamma=\operatorname{Cay}(V,\nos_q)$ defined in Definition~\ref{def:Cay(V_m,r,N_q)} with $S=\mathcal{L}$ and $T=\mathcal{L}'$. By Fact~\ref{fact:crucial-equality} and Lemma~\ref{lem:vanishing-subspace}, it is not hard to see that $e_{\nos_q}(\mathcal{L},\mathcal{L}')=\sum_{f\in\mathcal{L},f'\in\mathcal{L'}}\nos_q(f-f')$ and $\frac{1}{|V|}\sum_{f\in V}\nos_q(f)=\frac{1}{|V|}\sum_{\alpha\in\mathbb{F}_q^m}|V_\alpha|=q^{m-1}$. Moreover, by \cref{thm:spectrum}, the maximum absolute value of the non-trivial Fourier coefficients of $\nos_q$ is $q^{\dim V-1}$. 
    Putting all of this together, it follows from Lemma~\ref{lem:expander-mixing-Cayley} that
    \begin{equation*}
        \left|\sum_{f\in\mathcal{L},f'\in\mathcal{L'}}\nos_q(f-f')-q^{m-1}|\mathcal{L}||\mathcal{L}'|\right|\le q^{\dim V-1}\sqrt{|\mathcal{L}|\mathcal{|L'|}},
    \end{equation*}

    \noindent completing the proof of the theorem.
\end{proof}

\begin{proof}[Proof of \cref{thm:subpace-incidence-bound-multi-poly}]
    The upper bound follows directly from \cref{thm:cross-version} by setting $\mathcal{L}'=\mathcal{L}$, 
    and the lower bound follows from the Cauchy-Schwarz inequality:
    \begin{equation*}
        \begin{split}
            \sum_{f,f'\in\mathcal{L}}\nos_q(f-f')
            &=\sum_{\alpha\in\mathbb{F}_q^m,\beta\in\mathbb{F}_q}|\{f\in \mathcal{L}:f(\alpha)=\beta\}|^2\\
            &\geq\frac{1}{q^{m+1}}\left(\sum_{\alpha\in\mathbb{F}_q^m,\beta\in\mathbb{F}_q}|\{f\in\mathcal{L}:f(\alpha)=\beta\}|\right)^2=\frac{1}{q^{m+1}}\cdot (q^m|\mathcal{L}|)^2=q^{m-1}|\mathcal{L}|{^2}.
        \end{split}
    \end{equation*}
\end{proof}

\subsection{Point-multivariate polynomial incidences: Proof of \cref{thm:subspace-point-multi-poly}}

\noindent In this subsection, we prove \cref{thm:subspace-point-multi-poly} by the second moment method. Our proof is inspired by the work of Murphy and Petridis \cite{murphy2016point}, who applied a similar method to give a new proof of the point-line incidence bound of Vinh \cite{vinh2011szemeredi}.

\begin{proof}[Proof of \cref{thm:subspace-point-multi-poly}]
    For each $v\in\mathbb{F}_q^{m+1}$, let $i_{\mathcal{L}}(v)=|\{f\in\mathcal{L}:f(v_1,\ldots,v_{m})=v_{m+1}\}|$ be the number of polynomials in $\mathcal{L}$ that are incident to $v$. We bound the second moment of $i_{\mathcal{L}}(v)$ from the above as follows:
\begin{equation*}
    \begin{split}
        \sum_{v\in\mathbb{F}_q^{m+1}}i_{\mathcal{L}}(v)^2
        &=\sum_{v\in\mathbb{F}_q^{m+1}}|\{(f,f')\in\mathcal{L}\times\mathcal{L}:f(v_1,\ldots,v_m)=f'(v_1,\ldots,v_m)=v_{m+1}\}|\\
        &=\sum_{f,f'\in\mathcal{L}}|\{v\in\mathbb{F}_q^{m+1}:f(v_1,\ldots,v_m)=f'(v_1,\ldots,v_m)=v_{m+1}\}|\\
        &=\sum_{f,f'\in\mathcal{L}}\nos_q(f-f')\leq q^{m-1}|\mathcal{L}|^2+q^{\dim V-1}|\mathcal{L}|,
    \end{split}
\end{equation*}

\noindent where the last inequality follows from \cref{thm:subpace-incidence-bound-multi-poly}. 

We proceed to prove an upper bound on the variance of $i_{\mathcal{L}}(v)$:
\begin{equation*}
    \begin{split}
        \sum_{v\in\mathbb{F}_q^{m+1}}\left(i_{\mathcal{L}}(v)-\frac{|\mathcal{L}|}{q}\right)^2
        &=\sum_{v\in\mathbb{F}_q^{m+1}}i_{\mathcal{L}}(v)^2-\frac{2|\mathcal{L}|}{q}\sum_{v\in\mathbb{F}_q^{m+1}}i_{\mathcal{L}}(v)+q^{m+1}\frac{|\mathcal{L}|^2}{q^2}\\
        &=\sum_{v\in\mathbb{F}_q^{m+1}}i_{\mathcal{L}}(v)^2-\frac{2|\mathcal{L}|}{q}\cdot q^{m}|\mathcal{L}|+q^{m-1}|\mathcal{L}|^2\\
        &=\sum_{v\in\mathbb{F}_q^{m+1}}i_{\mathcal{L}}(v)^2-q^{m-1}|\mathcal{L}|^2\leq q^{\dim V-1}|\mathcal{L}|.
    \end{split}
\end{equation*}

Consequently, by the Cauchy-Schwarz inequality, 
\begin{equation*}
    \begin{split}
        \left|I(\mathcal{P},\mathcal{L})-\frac{|\mathcal{P}||\mathcal{L}|}{q}\right|
        &=\left|\sum_{v\in\mathcal{P}}\left(i_{\mathcal{L}}(v)-\frac{|\mathcal{L}|}{q}\right)\right|\\
        &\leq\sum_{v\in\mathcal{P}}\left|i_{\mathcal{L}}(v)-\frac{|\mathcal{L}|}{q}\right|\\
        &\leq\sqrt{|\mathcal{P}|\sum_{v\in\mathcal{P}}\left(i_{\mathcal{L}}(v)-\frac{|\mathcal{L}|}{q}\right)^2}\\
        &\leq\sqrt{|\mathcal{P}|\sum_{v\in\mathbb{F}_q^{m+1}}\left(i_{\mathcal{L}}(v)-\frac{|\mathcal{L}|}{q}\right)^2}\leq\sqrt{q^{\dim V-1}|\mathcal{P}||\mathcal{L}|},
    \end{split}
\end{equation*}

\noindent completing the proof of the theorem.
\end{proof}

\section{Concluding remarks}
\label{sec:5}

\noindent We investigated incidence problems involving multivariate polynomials of bounded degree over finite fields. Several interesting open questions arise from our work.

\paragraph{Closing the gap on the bounds of $\tau_q(m, r)$.} Recall that $\tau_q(m, r)$ denotes the smallest threshold such that for every $\mathcal{L} \subseteq V_{m,r}$ with $|\mathcal{L}| \gg \tau_q(m, r)$, $\sum_{f,f'\in\mathcal{L}} \nos_q(f - f') = (1 + o(1)) q^{m-1} |\mathcal{L}|^2$. Proposition~\ref{prop:incidence-bound-multi-poly} and the example below it showed that $q^{\dim V_{m,r-1}}\le \tau_q(m, r)\le q^{\dim V_{m,r}-m}$; in particular, $\tau_q(1, r)=q^r$ and $\tau_q(m,1)=q$. It is an interesting problem to determine the asymptotic order of $\tau_q(m, r)$ for each $m \ge 2$ and $r \ge 2$. 
\begin{question}
    What is the asymptotic order of $\tau_q(m, r)$ for $m \ge 2$ and $r \ge 2$? 
\end{question}

\paragraph{Improving upon the Cauchy–Schwarz bound for small $|\mathcal{P}|$ or $|\mathcal{L}|$.} Our point-multivariate polynomial incidence bound \eqref{eq:subspace-point-multi-poly}
\begin{equation*}
    \left|I(\mathcal{P}, \mathcal{L}) - \frac{|\mathcal{P}||\mathcal{L}|}{q} \right| 
    \leq q^{(\dim V - 1)/2} \sqrt{|\mathcal{P}||\mathcal{L}|}
\end{equation*}
improves upon the classical bound obtained via the Cauchy–Schwarz inequality \eqref{eq:point-multi-poly-CS}
\begin{equation*}    
    I(\mathcal{P}, \mathcal{L}) \leq \min \left\{ 
        |\mathcal{L}| + q^{\dim V_{m,r}/2 - 1} |\mathcal{P}| |\mathcal{L}|^{1/2},\quad 
        |\mathcal{P}| + r^{1/2} q^{(m - 1)/2} |\mathcal{P}|^{1/2} |\mathcal{L}|
    \right\}
\end{equation*}
in regimes where both $|\mathcal{P}|$ and $|\mathcal{L}|$ are large; for instance, when $|\mathcal{P}| > q$ and $|\mathcal{L}| > \frac{1}{r} q^{\dim V - m}$. It would be very interesting to beat the Cauchy–Schwarz inequality when $|\mathcal{P}|$ or $|\mathcal{L}|$ is relatively small (see e.g. \cite{bourgain2004sum, iosevich2024improved}). 

\begin{question}
    Can we improve upon the Cauchy–Schwarz bound \eqref{eq:point-multi-poly-CS} for small $|\mathcal{P}|$ or $|\mathcal{L}|$? Can we obtain even better incidence bounds by leveraging additional structural (algebraic or geometric) information of $|\mathcal{P}|$ or $|\mathcal{L}|$?
\end{question}

\paragraph{A lower bound on $\lambda$.} A crucial tool of our proof is the expander mixing lemma for abelian Cayley color graphs (Lemma~\ref{lem:expander-mixing-Cayley}), where the key parameter is the maximum absolute value of the non-trivial Fourier coefficients of the connection function. Below, we prove a lower bound for this quantity for Cayley color graphs. 

\begin{theorem}\label{thm:Cayley-Alon-Boppna}
    Let $G$ be a finite abelian group and let $c:G\to\mathbb{C}$. Let $\lambda$ be the largest non-trivial Fourier coefficient of $c$ in absolute value. Then $\lambda \ge \sqrt{|G|\operatorname{Var}_{g\sim G}|c(g)|}$. In particular, if $c$ is real-valued, then $\lambda \ge \sqrt{|G|\operatorname{Var}_{g\sim G}c(g)}$.
\end{theorem}

\begin{proof}
    By Plancherel's identity,
\begin{equation*}
    |G|\sum_{g\in G}|c(g)|^2
    = \sum_{\chi\in\widehat{G}}|\widehat{c}(\chi)|^2
    = \left|\widehat{c}(\chi_0)\right|^2+\sum_{\chi\neq\chi_0}|\widehat{c}(\chi)|^2
    \le \left|\sum_{g\in G}c(g)\right|^2+(|G|-1)\lambda^2.
\end{equation*}

\noindent Hence
\begin{equation*}
    \lambda^2 \ge \frac{|G|\sum_g |c(g)|^2 - |\sum_g c(g)|^2}{|G|-1} \ge \frac{|G|\sum_g |c(g)|^2 - (\sum_g |c(g)|)^2}{|G|},
\end{equation*}
where the second inequality follows from the triangle inequality. We rewrite the right-hand side of the second inequality in terms of the variance as follows: 
\begin{equation*}
\begin{split}
    \frac{|G|\sum_g |c(g)|^2 - (\sum_g |c(g)|)^2}{|G|}
    &= |G|\Biggl[\frac{1}{|G|}\sum_{g\in G}|c(g)|^2 - \Bigl(\frac{1}{|G|}\sum_{g\in G}|c(g)|\Bigr)^2\Biggr]\\
    &= |G|\Bigl(\mathbb{E}|c(g)|^2 - (\mathbb{E}|c(g)|)^2\Bigr)
    = |G|\operatorname{Var}|c(g)|.
\end{split}
\end{equation*}

\noindent It then follows that $\lambda^2 \ge |G|\,\operatorname{Var}_{g\sim G} |c(g)|$. If $c$ is real-valued, the same argument gives $\lambda^2\ge \frac{|G|\sum_g c(g)^2-(\sum_g c(g))^2}{|G|-1}\ge |G|\,\operatorname{Var}_{g\sim G}c(g)$.
\end{proof}

\section*{Acknowledgments}

\noindent C. Shangguan and Y. Yang are supported by National Natural Science Foundation of China under Grant Nos. 12571352 and 12231014, and Fundamental Research Funds for the Central Universities. T. Zhang is supported by National Natural Science Foundation of China under Grant No. 12571357, and Natural Science Basic Research Program of Shaanxi under Program No. 2025JC-YBMS-048. Part of the work was done when T. Zhang was visting Shandong University Qingdao Campus in May 2025.

{\small
\bibliographystyle{plain}
\normalem
\bibliography{ref}
}

\appendix

\section{Appendix}\label{sec-app}


\subsection{Incidence counting for Example~\ref{ex:counter-example}}\label{sec-app-A}



\noindent We compute the number of incidences directly:
\begin{equation*}
\begin{split}
I(\mathcal{P}_0, \mathcal{L}_0)
&=\sum_{\alpha \in \mathbb{F}_q^m} |\{f \in \mathcal{L}_0 : f(\alpha) = 0\}| \\
&=\sum_{\alpha \in \mathbb{F}_q^m: \alpha_1 = 0} |\{f \in \mathcal{L}_0 : f(\alpha) = 0\}| + \sum_{\alpha \in \mathbb{F}_q^m: \alpha_1 \ne 0} |\{f \in \mathcal{L}_0 : f(\alpha) = 0\}| \\
&= q^{m-1} |\mathcal{L}_0| + \sum_{\alpha\in\mathbb{F}_q^m:\alpha_1 \ne 0} |\{g \in V_{m, r-1} : g(\alpha) = 0\}| \\
&= q^{m-1} |\mathcal{L}_0| + (q^m - q^{m-1}) \cdot q^{\dim \mathcal{L}_0 - 1} \\
&= \left(2 - \frac{1}{q} \right) q^{m-1} |\mathcal{L}_0|.
\end{split}
\end{equation*}

\subsection{Proof of the Cauchy-Schwarz incidence bound \eqref{eq:point-multi-poly-CS}}\label{sec:app-B}

\noindent Here we present a proof of \eqref{eq:point-multi-poly-CS} for completeness. For each $v \in \mathbb{F}_q^{m+1}$, let 
$i_{\mathcal{L}}(v) = |\{ f \in \mathcal{L} : f(v_1, \ldots, v_m) = v_{m+1} \}|$ 
be the number of polynomials in $\mathcal{L}$ that are incident to $v$. Then $I(\mathcal{P},\mathcal{L})=\sum_{v \in \mathcal{P}} i_{\mathcal{L}}(v)$. Moreover, 
\begin{equation*}
    \begin{aligned}
        \sum_{v \in \mathcal{P}} i_{\mathcal{L}}(v)^2 
        & = \sum_{v \in \mathcal{P}} |\{ (f,f') \in \mathcal{L}\times\mathcal{L} : f(v_1,\ldots,v_m) = f'(v_1,\ldots,v_m) = v_{m+1} \}| \\
        & = I(\mathcal{P},\mathcal{L}) + \sum_{f,f' \in \mathcal{L},~f\neq f'} |\{ v \in \mathcal{P} : (f - f')(v_1, \ldots, v_m) = 0 \}|\\
        & \le I(\mathcal{P},\mathcal{L}) + r q^{m-1} |\mathcal{L}|^2,
    \end{aligned}
\end{equation*} 

\noindent where the last inequality follows from the Schwartz-Zippel lemma. By the Cauchy-Schwarz inequality, we have
\begin{equation*}
        \frac{I(\mathcal{P},\mathcal{L})^2}{|\mathcal{P}|}=\frac{1}{|\mathcal{P}|} \left( \sum_{v \in \mathcal{P}} i_{\mathcal{L}}(v) \right)^2 \leq \sum_{v \in \mathcal{P}} i_{\mathcal{L}}(v)^2 \le I(\mathcal{P},\mathcal{L}) + r q^{m-1} |\mathcal{L}|^2,
\end{equation*}

\noindent implying that $I(\mathcal{P},\mathcal{L}) \leq |\mathcal{P}| + r^{1/2} q^{(m-1)/2} |\mathcal{P}|^{1/2} |\mathcal{L}|$. 

It remains to show $I(\mathcal{P},\mathcal{L}) \leq |\mathcal{L}| + q^{\dim V_{m,r}/2 - 1} |\mathcal{P}| |\mathcal{L}|^{1/2}$. By symmetry, for each $f \in \mathcal{L}$, let 
$i_{\mathcal{P}}(f) = |\{ v \in \mathcal{P} : f(v_1, \ldots, v_m) = v_{m+1} \}|$ 
be the number of points in $\mathcal{P}$ that are incident to $f$. Then $I(\mathcal{P},\mathcal{L})=\sum_{f \in \mathcal{L}} i_{\mathcal{P}}(f)$. Moreover, 
\begin{equation*}
    \begin{aligned}
        \sum_{f \in \mathcal{L}} i_{\mathcal{P}}(f)^2 
        & = \sum_{f \in \mathcal{L}} |\{ (v,v') \in \mathcal{P}\times\mathcal{P} : f(v_1,\ldots,v_m) = v_{m+1},~f(v'_1,\ldots,v'_m) = v'_{m+1} \}| \\
        & = I(\mathcal{P},\mathcal{L}) + \sum_{v,v' \in \mathcal{P},~v\neq v'} |\{ f \in \mathcal{L} : f(v_1,\ldots,v_m) = v_{m+1},~f(v'_1,\ldots,v'_m) = v'_{m+1} \}|\\
        & \le I(\mathcal{P},\mathcal{L}) + q^{\dim V_{m,r}-2} |\mathcal{P}|^2,
    \end{aligned}
\end{equation*} 

\noindent where the last inequality follows from the fact that for distinct $v,v'$, $\{ f \in V_{m,r} : f(v_1,\ldots,v_m) = v_{m+1},~f(v'_1,\ldots,v'_m) = v'_{m+1} \}$ is either empty or forms an affine subspace of $V_{m,r}$ of codimension $2$. By the Cauchy-Schwarz inequality, one can similarly show that $I(\mathcal{P},\mathcal{L}) \leq |\mathcal{L}| + q^{\dim V_{m,r}/2 - 1} |\mathcal{P}| |\mathcal{L}|^{1/2}$. We omit the details.

\end{document}